\documentclass[11pt]{amsart}
\usepackage[a4paper]{geometry}                                              
\usepackage{fourier,mathtools}                                                       
\usepackage[bb=ams, cal=cm, scr=boondox, frak=euler]{mathalpha} 
\let\amsmathbb\mathbb           
\AtBeginDocument{%
    \let\mathbb\relax
    \newcommand{\mathbb}[1]{\amsmathbb{#1}}
}
\mathtoolsset{mathic=true}

\usepackage{amsmath,amsthm,amsfonts,amssymb,amscd}                               
\usepackage{mathrsfs}                                                       
\usepackage{esint}                                                          
\usepackage[all]{xy}                                                      
\usepackage[colorlinks=true,linkcolor=magenta,citecolor=blue]{hyperref}    
\usepackage[capitalise]{cleveref}                                          
\usepackage{colonequals}                                                   
\newcommand{\bP}{{\mathbb{P}}}
\newcommand{\bQ}{{\mathbb{Q}}}
\newcommand{\bZ}{{\mathbb{Z}}}
\newcommand{\cP}{{\mathcal{P}}}
\newcommand{\cU}{{\mathcal{U}}}
\newcommand{\et}{{\mathrm{\acute{e}t}}}
\newcommand{\sA}{{\mathscr{A}}}

\newcommand{\sF}{{\mathscr{F}}}
\newcommand{\sO}{{\mathscr{O}}}

\renewcommand{\phi}{\varphi}

\DeclareMathOperator{\CH}{CH}
\DeclareMathOperator{\Char}{char}

\DeclareMathOperator{\LN}{LN}

\DeclareMathOperator{\Tr}{Tr}
\DeclareMathOperator{\Ker}{Ker}

\newtheorem{thm}{Theorem}[section]           

\newtheorem{lem}[thm]{Lemma}

\theoremstyle{definition}

\usepackage{enumitem}

\begin{document}
\title{A specialisation theorem for Lang-N\'eron groups}                                                                              
\begin{abstract}
    We show that, for a polarised smooth projective variety $B \hookrightarrow \bP^n_k$ of dimension $\geq 2$ over an infinite field $k$ and an abelian variety $A$ over the function field of $B$, there exists a dense Zariski open set of smooth geometrically connected hyperplane sections $h$ of $B$ such that $A$ has good reduction at $h$ and the specialisation homomorphism of Lang-N\'eron groups at $h$ is injective (up to a finite $p$-group in positive characteristic $p$). This gives a positive answer to a conjecture of the first author, which is used to deduce a negative definiteness result on his refined height pairing. This also sheds a new light on N\'eron's specialisation theorem.
\end{abstract}

\author{Bruno Kahn}
\address{Sorbonne Universit\'e and Universit\'e Paris Cit\'e, CNRS, IMJ-PRG, F-75005 Paris, France.}
\email{bruno.kahn@imj-prg.fr}

\author{Long Liu}
\address{Sorbonne Universit\'e and Universit\'e Paris Cit\'e, CNRS, IMJ-PRG, F-75005 Paris, France.}
\email{long.liu@imj-prg.fr}  
\date{October 8, 2024}
\maketitle  

\tableofcontents

\section{Introduction}
Let $K/k$ be a finitely generated regular extension of fields of exponential characteristic $p$, and let $A$ be an abelian variety over $K$. Then $A$ has a $K/k$-trace $C=\Tr_{K/k} A$: this operation is right adjoint to the extension of scalars of abelian varieties. A celebrated theorem of Lang and Néron says that the \emph{Lang-Néron group}
\[\LN(K/k,A)=A(K)/C(k)\]
is finitely generated (\cite{LangNeron59LNfg}, {see also} \cite{Conrad06ChowLN} {and \cite{Kahn09LangNeronShiodaTate}}).

Let $B$ be a smooth model of $K/k$, and let $h\in B$ be a point of codimension $1$ whose residue field $E$ is also regular over $k$. If $A$ has good reduction at $h$, write $A_h$ for its reduction and $C_h=\Tr_{E/k}(A_h)$.
 Then we have a $k$-morphism 
\begin{align}\label{spTr} 
    \varphi_0 \colon C \longrightarrow C_h
\end{align}
and a commutative diagram of specialisation maps with exact rows:
\begin{equation}\label{sp}
    \begin{gathered}
        \xymatrix{
            0 \ar[r] & C(k)   \ar[r] \ar[d]^{\varphi_0(k)} & A(K)   \ar[r] \ar[d]^{\varphi} & \LN(A,K/k)   \ar[r] \ar[d]^{\psi} &0\\
            0 \ar[r] & C_h(k) \ar[r]                       & A_h(E) \ar[r]                  & \LN(A_h,E/k) \ar[r]               &0
        }
    \end{gathered}
\end{equation}
see  \cite[\S6B]{Kahn24RefinedHeight}. Recall their construction: the stalk $R \colonequals \sO_{B,h}$ is a discrete valuation ring with field of fractions $K$ and, by hypothesis, $A$ has a smooth proper model $\sA$ over $R$. Then $\sA$ is automatically an abelian scheme and is the N\'eron model of $A$ by \cite[\S 1.4, Proposition~2]{BLR90Neron}. Then $A_h$ is the special fibre of $\sA$, which is an abelian variety over $E$; this yields $\varphi$. The co-unit $C_K \to A$ extends to a morphism $C_{R} \to \sA$, and induces a morphism $C_E \to A_h$, thus the existence of $\varphi_0$ follows from the universal property of $C_h$.

\begin{thm}\label{spLNthm}
Assume that $B$ is smooth projective of dimension $d \geq 2$. For any projective embedding $B \hookrightarrow \bP^n_k$, there exists a dense open subset \(\cU\) of the dual projective space \(\cP\) of \(\bP^n_k\) such that if $H$ lies in \(\cU(k)\), then 
    \begin{enumerate}[label={\rm(\alph*)}]
        \item the hyperplane section $h \colonequals H \cap B$ is smooth geometrically connected of dimension $d-1$, 
        \item $A$ has good reduction at $h$,
        \item  in Diagram \eqref{sp}, the kernels of all vertical maps are finite $p$-groups (hence these maps are injective in characteristic $0$),
        \item the map $\varphi_0$ of \eqref{spTr} is a $p$-isogeny (an isomorphism in characteristic $0$).
    \end{enumerate}
\end{thm}

(If $k$ is infinite, so is \(\cU(k)\). When $k$ is finite, \(\cU(k)\) may be empty because in general there are no smooth hyperplane sections in $B$ defined over $k$; this issue can presumably be solved by composing the given projective embedding with a suitable Veronese embedding (see \cite[Corollary 1.6]{Gabber01SpaceFillingCurvesAlb} and \cite[Theorem~3.1]{Poonen04BertiniFq}).)

Besides Bertini's theorem, our main tool is a form of the weak Lefschetz theorem due to Deligne \cite[A.5]{Katz93AffinePerveMonodromy}, which renders the proof almost trivial.

The first application is to a negative definiteness result for the height pairing introduced in  \cite{Kahn24RefinedHeight}. For a smooth projective variety $X$ of dimension $d$ over $K$ and $i\in [0,d]$, the first author defined a subgroup $\CH^i(X)^{(0)}$  of the $i$-th Chow group of $X$ and a pairing
\[ \CH^i(X)^{(0)} \times \CH^{d+1-i}(X)^{(0)} \to \CH^1(B)\otimes \bQ. \]
For $i=1$, this  pairing induces a quadratic form on the Lang-N\'eron group of the Picard variety of $X$.  In \cite[Theorem~6.6]{Kahn24RefinedHeight}, it is proven that this quadratic form is negative definite if $B$ is a curve, and that one can reduce to this case when $\dim B >1$ if $\psi$ has finite kernel in \eqref{sp} for a suitable $h$  \cite[Conjecture~6.3]{Kahn24RefinedHeight}. Thus Theorem \ref{spLNthm} proves this conjecture\footnote{At least for $k$ infinite, but this is sufficient for the application: see \cite[part (d) of the proof of Theorem 6.6]{Kahn24RefinedHeight}.} (in a stronger form, and without the hypothesis of semi-stable reduction appearing in loc. cit.).

The second application is to N\'eron's specialisation theorem. Assume that $k$ is a number field.  If $B=\bP^n_k$ and $U$ is an open subset of $B$ over which $A$ extends to an abelian scheme $\sA$, then the set of rational points $t\in U(k)$ such that the specialisation map $A(K)\to \sA_t(k)$ is not injective is thin (\cite[11.1, theorem]{Serre97MordellWeil},  see \cite{CT20MW} for generalisations). The injectivity of $\varphi$ in Theorem \ref{spLNthm}  gives a version of this specialisation result which does not involve Hilbert's irreducibility theorem, but of course requires $\dim B>1$; see however Remark \ref{rmk}.

\section{Auxiliary results}
We start with the following standard lemmas.

\begin{lem}\label{ExtendRatPt}
    Let $U$ be an integral normal noetherian scheme with function field $K$. Let $\sA$ be an abelian scheme over $U$  with generic fibre $A$. Then the pull-back map
    \[ \sA(U) \to A(K) \]
    is an isomorphism.
\end{lem}

\begin{proof}
This is a consequence of the valuative criterion of properness and Weil's extension theorem (\cite[Proposition 1.3]{Artin86NeronModels} or \cite[\S 4.4, Theorem 1]{BLR90Neron}).
\end{proof}

\begin{lem}
    Let $U$ be a scheme and let $\sA$ be an abelian scheme over $U$. If $n$ is invertible on $U$, i.e., $n$ is prime to $\Char(k(x))$ for all $x\in U$, then we have an injection
    \[ \sA(U)\big/n \hookrightarrow H^1_\et(U,{}_n\sA),\]
    where ${}_n\sA$ is the kernel of the multiplication by $n$ on $\sA$.
\end{lem}
\begin{proof}
    Use the short exact sequence of \'etale sheaves
    \[ 0 \to {}_n \sA \to \sA \stackrel{n}{\to} \sA \to 0. \qedhere\]
\end{proof}

By the above lemmas, we can use cohomology to study the specialisation of $A(K)$. We shall rely on the following version of the weak Lefschetz theorem.

\begin{thm}[Deligne; see {\cite[Corollary A.5]{Katz93AffinePerveMonodromy}}]\label{DeligneWeakLef}
    Let $k$ be a separably closed field and let $\ell \neq \Char(k)$ be a prime. Let $f \colon U \to \bP^n_k$ be a separated quasi-finite morphism and let $\sF$ be a lisse $\overline{\bQ_\ell}$-sheaf on $U$. Assume that $U$ is smooth over $k$ and is of pure dimension $d$. Then there exists a dense open subset \(\cU\) of the dual projective space \(\cP\) of \(\bP^n_k\) such that if $H$ lies in \(\cU\), then the restriction map
    \[ H^i(U,\sF) \longrightarrow H^i(f^{-1}(H), \sF|_{f^{-1}(H)} )\]
    is an isomorphism for $i< d-1$ and injective for $i=d-1$.
\end{thm}
\begin{proof}
    In fact, in loc. cit., this theorem is proven when $k$ is algebraically closed for general perverse sheaves without assuming that $U$ is smooth and is of pure dimension $d$. In our case $\sF[d]$ is a perverse sheaf; see \cite[p. 139]{KW01Weil}. Moreover, the algebraically closed case immediately implies the separably closed case.
\end{proof}

\section{Proof of Theorem \ref{spLNthm}}\label{s3}
Choose a dense open subset $U$ of $B$ and an abelian scheme $\sA$ over $U$ such that $A \simeq \sA_K$ (see \cite[Remark 20.9]{Milne86AV}). 
Applying Bertini's theorem \cite[Corollary 6.11(2)]{Jouanolou83Bertini} to $B$ and $U$, we get a dense open subset \(\cU_1\) of the dual projective space $\cP$ of $\bP^n_k$ such that if $H$ lies in \(\cU_1(k)\) then $B \cap H$ (hence $U \cap H$) is smooth and geometrically connected of dimension $d-1$, and $U \cap H\neq \emptyset$. In particular, $A$ has good reduction at $B \cap H$ if $H \in \cU_1(k)$. 

Let $\ell$ be a prime different from $p$. Then by \cite[\S 7.3, Lemma 2]{BLR90Neron}, the kernel ${}_{\ell^m}\sA$ of multiplication by $\ell^m$ on $\sA$ is finite and \'etale. Thus it represents a locally constant constructible \'etale sheaf on $U$. Denote by $T_\ell\sA$ the lisse $\ell$-adic sheaf $({}_{\ell^m}\sA)$. 
    
Let $k_s$ be a separable closure of $k$. We denote base change from $k$ to $k_s$ by an index $s$. Note that the immersion $f \colon U \hookrightarrow \bP_k^n$ induced by the projective embedding $B \hookrightarrow \bP_k^n$ is separated quasi-finite.  By Theorem~\ref{DeligneWeakLef}, there exists a dense open subset \(\cU_2\) of the dual projective space $\cP_s$ such that if $H$ lies in \(\cU_2\), then the restriction map
\[ H^i(U_s,T_\ell\sA) \otimes_{\bZ_\ell} \overline{\bQ_\ell} \longrightarrow H^i(U_s \cap H, T_\ell\sA) \otimes_{\bZ_\ell} \overline{\bQ_\ell}\]
is an isomorphism for $i< d-1$ and injective for $i=d-1$. Therefore the restriction map
\[ H^i(U_s,T_\ell\sA) \longrightarrow H^i(U_s \cap H, T_\ell\sA)\]
has finite kernel and cokernel  for $i< d-1$ and finite kernel for $i=d-1$. (Recall that $H^i_\et(U_s,{}_{\ell^m}\sA) $ is finite for all $m$ by \cite[Th. finitude]{SGA4.5}, hence $H^i(U_s,T_\ell\sA) $ is a finitely generated $\bZ_\ell$-module.)

The open subset $\cU_2$ is defined over a finite Galois extension of $k$; taking the intersection of its conjugates, we may assume that it is defined over $k$. Take \(\cU = \cU_1 \cap\cU_2\). 

We now proceed in three steps:
    
\subsection{$\Ker \psi$ is finite}   \label{finiteKer}
    For $H\in {\cU(k)}$, we write $h=B\cap H$. Since the groups $C(k_s)$ and $C_h(k_s)$ are $\ell$-divisible, we have the isomorphisms
    \[ \sA(U_s)\big/\ell^m \simeq \left(\sA(U_s)/C(k_s)\right) \otimes_\bZ \bZ/\ell^m\bZ \simeq \LN(A,Kk_s/k_s)\big/\ell^m,\]
    where the second one holds by Lemma~\ref{ExtendRatPt}. Similarly, we have such isomorphisms for $\LN(A_h,Ek_s/k_s)$.
    Taking the inverse limit of the following commutative diagrams
    \[\xymatrix{
        H^1_\et(U_s,{}_{\ell^m}\sA)       \ar[d] & \sA(U_s)\big/\ell^m        \ar@{_(->}[l] \ar[r]^-\sim \ar[d]  & \LN(A,Kk_s/k_s)\big/\ell^m   \ar[d]\\
        H^1_\et((U\cap h)_s,{}_{\ell^m}\sA)        & \sA((U\cap h)_s)\big/\ell^m  \ar@{_(->}[l] \ar[r]^-\sim         & \LN(A_h,Ek_s/k_s)\big/\ell^m
    }\]
    we get the commutative diagram
    \[\xymatrix{
        H^1(U_s,T_\ell\sA)      \ar[d] & \sA(U_s)^{\wedge}      \ar@{_(->}[l] \ar[r]^-\sim \ar[d] & \LN(A,Kk_s/k_s)^{\wedge}  \ar[d]^{\psi^{\wedge}} \\
        H^1((U\cap h)_s,T_\ell\sA)      & \sA((U\cap h)_s)^{\wedge} \ar@{_(->}[l] \ar[r]^-\sim        & \LN(A_h,Ek_s/k_s)^{\wedge}
    }\]
    where $(-)^{\wedge}$ denotes $\ell$-adic completion. Since the left vertical arrow has finite kernel, so do the others. By the Lang-N\'eron theorem, the abelian group $\LN(A,K/k)$ is finitely generated. Thus $\psi^{\wedge}=\psi\otimes\bZ_\ell$, which implies that $\psi$ has a finite kernel. But $\LN(A,K/k)$ injects into $\LN(A,Kk_s/k_s)$, so we are done.

\subsection{The map $\varphi_0$ is an isogeny}\label{spTrIsogeny}
    Given an abelian group $M$, we have the $\ell$-adic Tate module of $M$ defined by
    \[ T_\ell M = \varprojlim_m ({}_{\ell^m}M), \]
    where ${}_{\ell^m}M$ is the $\ell^m$-torsion of $M$ and the inverse limit is over positive integers $m$ with transition morphisms given by the multiplication-by-$\ell$ map ${}_{\ell^{m+1}}M \to {}_{\ell^m}M$. The Tate modules of the Lang-N\'eron groups vanish because the Lang-N\'eron groups are finitely generated and the transition map is multiplication by $\ell$. Consider the commutative diagram
    \[\xymatrix{
        T_\ell(C(k_s))   \ar[r] \ar[d]^{T_\ell(\varphi_0)} & T_\ell(A(Kk_s))   \ar[d]^{T_\ell(\varphi)} & H^0(U_s,T_\ell \sA) \ar[d]\ar[l] \\
        T_\ell(C_h(k_s)) \ar[r]                            & T_\ell(A_h(Ek_s))                        & H^0((U\cap h)_s, T_\ell \sA). \ar[l]  
    }\]

    Applying  the (left exact) Tate module functor to Diagram \eqref{sp}, we see that the left horizontal maps are isomorphisms;  Lemma \ref{ExtendRatPt} then shows the same for the right horizontal maps.
    
    Applying Theorem \ref{DeligneWeakLef} to the right vertical arrow, we obtain that all vertical arrows have finite kernels and cokernels. By \cite[Remark 8.4]{Milne86AV}, $T_\ell(C(k_s))$ (resp. $T_\ell(C_h(k_s))$) is a free $\bZ_\ell$-module whose rank is $2\dim C$ (resp. $2\dim C_h$). It follows that the left vertical arrow is injective, and that $C$ and $C_h$ have the same dimension. The injectivity of $T_\ell(\varphi_0)$ implies that the $\ell$-adic Tate module of the abelian variety ${(\ker \varphi_0)}_{\rm red}^0$ vanishes. Thus ${(\ker \varphi_0)}_{\rm red}^0$ is a $0$-dimensional abelian variety and $\ker\varphi_0$ is a finite group scheme. Hence $\varphi_0$ is an isogeny.   

\subsection{End of proof}\label{pGrp}
    So far, we have proven that the kernels of the vertical maps in \eqref{sp} are finite, and it remains to show that they are $p$-primary. Since the formation of $K/k$-trace commutes with base change \cite[Th. 6.8]{Conrad06ChowLN} and $\LN(A,K/k)$ injects into $\LN(A,Kk_s/k_s)$ (ibid., proof of Lemma 7.3), we reduce to the case $k=k_s$. But then,  the map $\varphi \colon A(K) \to A_h(E)$ is injective on $n$-torsion for any $n$ invertible in $k$ \cite[p. 153]{Serre97MordellWeil} and $C_h(k)$ is $n$-divisible for such $n$ \cite[Th. 8.2]{Milne86AV}. Thus the conclusion follows by applying the snake lemma to  Diagram \eqref{sp}.

\subsection{Remark}\label{rmk}
    Let $x$ be a closed point of $B$, let $A_x$ be the special fiber of $\sA$ at $x$, and let $R_x$ be the Weil restriction of $A_x$ through the finite extension $k(x)/k$. The map $C_U \to \sA$ induces a map $C_{k(x)} \to A_x$ and then induces a map $C \to R_x$ by functoriality. Consider the following commutative diagram with exact rows
    \[\xymatrix{
        0 \ar[r] & C(k)   \ar[r] \ar[d]^{\varphi_0} & A(K)      \ar[r] \ar[d]^{\varphi} & \LN(A,K/k)   \ar[r] \ar[d]^{\psi} &0\\
        0 \ar[r] & R_x(k) \ar[r]^-{\sim}            & A_x(k(x)) \ar[r]                  & 0            \ar[r]               &0.
    }\]
    
    The snake lemma gives us an exact sequence 
    \begin{equation*}
        0 \to \Ker\varphi_0 \to \Ker\varphi \to \LN(A,K/k) \to A_x(k(x))/C(k)
    \end{equation*}
    
    A similar argument to \S \ref{pGrp} shows that $\Ker\varphi_0$ is a finite $p$-group. Thus $\Ker \varphi$ is finitely generated and its rank is uniformly bounded when $x$ varies.


\end{document}